\documentclass[12pt]{amsart}
\usepackage{txfonts}
\usepackage{amssymb}
\usepackage{eucal}
\usepackage{graphicx}
\usepackage{amsmath}
\usepackage{amscd}
\usepackage[all]{xy}           

\usepackage{amsfonts,latexsym}
\usepackage{xspace}
\usepackage{epsfig}
\usepackage{float}
\usepackage{color}
\usepackage{fancybox}
\usepackage{colordvi}
\usepackage{multicol}
\usepackage{colordvi}
\usepackage{stmaryrd}

\newtheorem{theorem}{Theorem}[section]
\newtheorem{defn}[theorem]{Definition}
\newtheorem{lemma}[theorem]{Lemma}

\newtheorem{prop-def}{Proposition-Definition}[section]

\newcommand{\nc}{\newcommand}
\newcommand{\delete}[1]{}

\nc{\mlabel}[1]{\label{#1}}  
\nc{\mcite}[1]{\cite{#1}}  
\nc{\mref}[1]{\ref{#1}}  
\nc{\mbibitem}[1]{\bibitem{#1}} 

\delete{
\nc{\mlabel}[1]{\label{#1}  
{\hfill \hspace{1cm}{\bf{{\ }\hfill(#1)}}}}
\nc{\mcite}[1]{\cite{#1}{{\bf{{\ }(#1)}}}}  
\nc{\mref}[1]{\ref{#1}{{\bf{{\ }(#1)}}}}  
\nc{\mbibitem}[1]{\bibitem[\bf #1]{#1}} 
}

\nc{\bfk}{\mathbf{k}}
\nc{\Der}{\mathrm{Der}}
\nc{\Ker}{\mathrm{Ker}}

\begin{document}

\title{The infinite dimensional  Unital 3-Lie Poisson algebra}\footnotetext{ Corresponding author: Ruipu Bai, E-mail: bairuipu@hbu.edu.cn.}

\author{Chuangchuang Kang}
\address{College of Mathematics and Information Science,
Hebei University, Baoding 071002, P.R. China}
         \email{kangchuang2016@163.com}

         \author{Ruipu  Bai}
\address{College of Mathematics and Information Science,
Hebei University,
Key Laboratory of Machine Learning and Computational Intelligence of Hebei Province, Baoding 071002, P.R. China} \email{bairuipu@hbu.edu.cn}

\author{Yingli Wu}
\address{College of Mathematics and Information Science,
Hebei University, Baoding 071002, P.R. China}
         \email{15733268503@163.com}

\date{}

\begin{abstract}
From a commutative associative algebra $A$, the  infinite dimensional unital 3-Lie Poisson algebra~$\mathfrak{L}$~is constructed, which is also a canonical Nambu 3-Lie algebra, and the structure of $\mathfrak{L}$ is discussed. It is proved that:  (1) there is a minimal set of generators $S$ consisting of six vectors; (2) the quotient algebra $\mathfrak{L}/\mathbb{F}L_{0, 0}^0$ is a simple  3-Lie Poisson algebra;
(3) four important infinite dimensional  3-Lie algebras: 3-Virasoro-Witt algebra $\mathcal{W}_3$, $A_\omega^\delta$,  $A_{\omega}$ and the 3-$W_{\infty}$ algebra can be embedded in $\mathfrak{L}$.
\end{abstract}

\keywords{ 3-Lie algebra, 3-Lie Poisson algebra, canonical Nambu 3-Lie algebra}

\maketitle
\vspace{-.5cm}


\numberwithin{equation}{section}

\allowdisplaybreaks

\section{Introduction}

The role of $n$-Lie Poisson algebra in string theory has received more and more attention in recent years(\cite{Axenides,Sochichiu,Azc¨¢rraga,Papadopoulos}).~For example, the applications of $n$-Lie Poisson algebra
  in
  noncommutative geometry and  the quantum geometry of branes in M-theory have been considered in \cite{DSS}. The theory of~$n$-Lie Poisson algebras provides an useful method to describe multiple M2 branes (\cite{JN,JN2,JN3}), and algebraical and geometrical structures of Nambu-Poisson manifold (\cite{DeBellis,Das,Vallejo}).

The discovery of 3-bracket in 1973 triggered a huge amount of innovative scientific inquiry. In \cite{Nambu}, N. Nambu  first proposed the notion of 3-bracket for constructing the  generalized Hamiltonian dynamics. In \cite{Takh}, Takhtajan studied the  algebraic structures in Nambu mechanics, and indicated the relation between Nambu mechanics and $n$-Lie algebras. In 1985, Filippov introduced  $n$-Lie algebras (\cite{Filippov}), and then the structures are studied  in \cite{Hanlon,Loday,Michor,Pozhidaev1,Pozhidaev2,Ling}. These earlier studies demonstrate a strong and consistent association between  $n$-Lie algebras and $n$-Lie Poisson algebras.  The $n$-Lie Poisson algebra comes from the concept of generalized Poisson ($n$-Poisson or Nambu-Poisson) structure which is naturally defined for $n$-brackets with an even number of entries,  and parallels the properties of higher order generalized Lie algebras (or $n$-Lie algebras)(\cite{Azc¨¢rraga2}).

An $n$-Lie Poisson algebra is an associative commutative algebra with a totally antisymmetric $n$-Lie bracket satisfying  the generalized Leibniz rule and the fundamental identity (FI) given in \cite{Filippov}. Actually, removing the property of Leibniz rule requirement for the $n$-Lie Poisson algebras, while keeping the fundamental identity, we get $n$-Lie algebras. Further, if the $n$-bracket need not be anticommutative, we can get another useful algebras which are called $n$-Leibniz algebras (or $n$-Loday algebras) (\cite{Daletskii,Casas}). Physically, the fundamental identity is a consistency condition for the time evolution,~which is given in terms of $(n-1)$ ¡®Hamiltonian¡¯ functions that determines derivations of $n$-Lie algebras(\cite{Sahoo}).

However,~the main challenge faced by  researchers is how to find  $n$-Lie Poisson algebras. Since the multiple multiplication and Leibniz rule should be satisfied simultaneously, the structure of $n$-Lie Poisson algebras is more complicated than that of $n$-Lie algebras. So the construction of  3-Lie Poisson algebras as the special case of
$n$-Lie Poisson algebras is very important.

In \cite{CTDFC}, Curtright constructed a 3-Virasoro-Witt algebra through the use of $su(1,1)$ enveloping algebra techniques. In \cite{BR-2}, Bai provided two methods for  constructing  infinite dimensional 3-Lie algebras $A_{\omega}$ and $A_{\omega}^{\delta}$ from group algebras. In \cite{CSAS},
Chakrabortty  obtained $W_{\infty}$ 3-algebras by using ``lone-star'' product of generators in $W_{\infty}$-algebras as well as their commutation relations, and appropriate double scaling limits.
Connected strongly to physics,  the 3-Lie algebras constructed as above are very important.

The purpose of this paper is to construct  an infinite dimensional 3-Lie Poisson algebra $\mathfrak{L}$ which contains 3-Virasoro-Witt algebra $\mathcal{W}_3$, 3-Lie algebras $A_{\omega}$,  $A_{\omega}^{\delta}$,  and $W_{\infty}$ 3-algebras  simultaneously.

In Section 2, we introduce some basic notions.
In Section 3, we provide an infinite dimensional  unital 3-Lie Poisson algebra $\mathfrak L$ and realize it by canonical Nambu $3$-Lie algebras.
Section 4 we pay close attention to the Lie structure of unital 3-Lie Poisson algebra  $\mathfrak L$  and study four types of 3-Lie subalgebras.
Section 5 is devoted to derivations of unital 3-Lie Poisson algebra.

Unless otherwise stated,  algebras and vector spaces are over a field $\mathbb{F}$ of characteristic zero,  $\mathbb{Z}$ is the set of integers, $\mathbb{Z}^{+}$ is the set of positive integers, and $\mathbb R$ is the real field.
For any  algebra $A$, and $S\subseteq A$, we use $\langle S\rangle$ to denote the subalgebra of $A$ generated by $S$.

\section{Preliminary}
\setcounter{equation}{0}
\renewcommand{\theequation}
{2.\arabic{equation}}

\begin{defn}\cite{Filippov}\label{defn:nLiealgebra} An {\bf $n$-Lie algebra }$A$ is a vector space over $ \mathbb F$ endowed with an $n$-ary multi-linear skew-symmetric multiplication $[ ,\cdots , ]$ satisfying  Fundamental Identity (FI), for all $x_1,x_2,\cdots,x_n,y_2,\cdots,y_n\in A$,
\begin{equation}\label{eq:Jaco}
[[ x_1,\cdots, x_n],y_2,\cdots,y_n]=\sum\limits_{i=1}^n[x_1, \cdots [x_i, y_2,\cdots,y_n], \cdots x_n].
\end{equation}
\end{defn}

Let $A$ be an $n$-Lie algebra. For any $x_1, \cdots, x_{n-1}\in A$, the linear mapping $\mbox{ad}(x_1, \cdots, x_{n-1}): A\rightarrow A$ defined by
\begin{equation}\label{eq:left}
\mbox{ad}(x_1, \cdots, x_{n-1})(x)=[x_1, \cdots, x_{n-1}, x], ~~ \forall x\in A,
\end{equation}
is called {\bf the left multiplication determined by $x_1, \cdots, x_{n-1}$}.  Thanks to \eqref{eq:Jaco}, left multiplications are derivations.

\begin{defn}\cite{DA}\label{defn:Jacobian algebra}
An  {\bf $n$-Lie Poisson algebra} (or  Nambu-Poisson algebra) $(U,\cdot,[ , \cdots,])$ over a field $\mathbb F$ is a linear vector space $U$ with $\mathbb F$-linear multiplications  $\cdot : U\otimes U\rightarrow U$, and $[ , \cdots, ]: \wedge^nU\rightarrow U$ satisfying
\begin{itemize}
\item $(U,\cdot)$ is an associative commutative algebra;
\item $(U, [ , \cdots,])$ is an n-Lie algebra;
\item the following Leibniz rule holds:
\begin{equation}\label{eq:Lei1}
[ x_1\cdot x_2,y_2,\cdots,y_n]=[x_1,y_2,\cdots,y_n]\cdot x_2+x_1\cdot [x_2,y_2,\cdots,y_n],\forall~ x_1,x_2,y_2,\cdots,y_n\in U.
\end{equation}
\end{itemize}
\end{defn}
If there is an unit element in $(A, \cdot)$, then  $A=(U,\cdot,[ , \cdots,])$  is called {\bf an unital $n$-Lie Poisson algebra}.

An $n$-Lie Poisson algebra $A=(U,\cdot, [, \cdots, ])$ is usually denoted by  $A=(U,[ , \cdots,])$ or $A$, for all $u, v\in U$, $u\cdot v$ is denoted by $uv.$

Recall  \textbf{Jacobian algebras} given in \cite{BR-2}.
Let $J$ be a commutative associative algebra,  $D_1,\cdots,D_n$ be  commutative  derivations of $J$, then  the multiplication on $J$ defined by
\begin{equation}\label{eq:Jacobian}
  [x_1,\cdots,x_n]=det\left(
                        \begin{array}{ccc}
                          D_1(x_1) & \cdots & D_1(x_n) \\
                          \cdots   & \cdots & \cdots \\
                          D_n(x_1) & \cdots & D_n(x_n) \\
                        \end{array}
                      \right),~~\forall~ x_1,x_2,...,x_n\in J,
\end{equation}
satisfies Definition \ref{defn:Jacobian algebra}. Therefore, $(J, [ ,\cdots, ])$  is an $n$-Lie Poisson algebra, and it is called the Jacobian algebra defined by  commutative  derivations $\{D_1,\cdots,D_n\}$.

\section{Unital 3-Lie Poisson algebra}
\setcounter{equation}{0}
\renewcommand{\theequation}
{3.\arabic{equation}}
Let $A$ be a commutative associative algebra with  a basis $\{L_{l,m}^{r} ~|~ l,m,r\in \mathbb{Z}+\frac{1}{2}\mathbb{Z}\}$, and the multiplication in the basis be as follows
\begin{equation}\label{eq:muit}
  L_{l_1,m_1}^{r_1}\cdot L_{l_2,m_2}^{r_2}=L_{l_1+l_2,m_1+m_2}^{r_1+r_2},~~\forall~ r_i,l_i,m_i\in \mathbb{Z}+\frac{1}{2} \mathbb{Z},~~1\leq i\leq 2.
\end{equation}

Since for all $r,l,m\in\mathbb{Z}+\frac{1}{2} \mathbb{Z} $,
\begin{equation}\label{eq:unit}
  L_{0,0}^{0}\cdot L_{l,m}^{r}=L_{l,m}^{r},
\end{equation}
$L_{0,0}^{0}$ is the only unit element of $A$.

For convenience, for all   $\alpha_1=(l_1, m_1, r_1),~ \alpha_2=(l_2, m_2, r_2),~ \alpha_3=(l_3, m_3, r_3) \in (\mathbb{Z}+\frac{1}{2} \mathbb{Z})^{ 3}$, $ M(\alpha_1,\alpha_2,\alpha_3)$ denotes the determinant
\begin{equation}\label{eq:r1l1m1}
 M(\alpha_1,\alpha_2,\alpha_3):=\left|
\begin{array}{ccc}
 r_1 & r_2 & r_3 \\
l_1 & l_2 & l_3 \\
 m_1 & m_2 & m_3 \\
\end{array}
\right|, ~~\mbox {and }~~ \nu=(1, 1, 0)\in (\mathbb{Z}+\frac{1}{2} \mathbb{Z})^{ 3}.
\end{equation}

Define the 3-ary multiplication $[ , , ]: A^{\wedge 3}\rightarrow A$ as follows,
$ \forall~~ \alpha_1=(l_1, m_1, r_1)$, $\alpha_2=(l_2, m_2, r_2)$, $\alpha_3=(l_3, m_3, r_3)\in (\mathbb{Z}+\frac{1}{2} \mathbb{Z})^{ 3}$,
\begin{equation}\label{eq:3-Jacobi bracket}
[L_{l_1,m_1}^{r_1},L_{l_2,m_2}^{r_2},L_{l_3,m_3}^{r_3}]:= M(\alpha_1,\alpha_2,\alpha_3)L_{l_1+l_2+l_3-1,m_1+m_2+m_3-1}^{r_1+r_2+r_3}.
\end{equation}

By the above notations,  we have the following result.
\begin{theorem}\label{thm:3-Jacobi}   $(A,\cdot,[ , , ])$ is an unital 3-Lie Poisson algebra.
\end{theorem}

\begin{proof} We only need to prove that \eqref{eq:Jaco} and \eqref {eq:Lei1} hold.  By \eqref{eq:3-Jacobi bracket}, for all  $ \alpha_i=(l_i, m_i, r_i)\in (\mathbb{Z}+\frac{1}{2}\mathbb{Z})^{3},~ 1\leq i\leq 5$,

\begin{equation*}
\begin{split}
&[[L_{l_1,m_1}^{r_1},L_{l_2,m_2}^{r_2},L_{l_3,m_3}^{r_3}],L_{l_4,m_4}^{r_4},L_{l_5,m_5}^{r_5}]\\
=&[M(\alpha_1,\alpha_2, \alpha_3)L_{l_1+l_2+l_3-1,m_1+m_2+m_3-1}^{r_1+r_2+r_3},L_{l_4,m_4}^{r_4},L_{l_5,m_5}^{r_5}]\\\\
=&M(\alpha_1,\alpha_2, \alpha_3)M(\alpha_1+\alpha_2+\alpha_3-\nu, \alpha_4,\alpha_5)L_{l_1+l_2+l_3+l_4+l_5-2,m_1+m_2+m_3+m_4+m_5-2}^{r_1+r_2+r_3+r_4+r_5},
\end{split}
\end{equation*}

\begin{equation*}
\begin{split}
&[[L_{l_1,m_1}^{r_1},L_{l_4,m_4}^{r_4},L_{l_5,m_5}^{r_5}],L_{l_2,m_2}^{r_2},L_{l_3,m_3}^{r_3}]
+[L_{l_1,m_1}^{r_1},[L_{l_2,m_2}^{r_2},L_{l_4,m_4}^{r_4},L_{l_5,m_5}^{r_5}],L_{l_3,m_3}^{r_3}]\\
+&[L_{l_1,m_1}^{r_1},L_{l_2,m_2}^{r_2},[L_{l_3,m_3}^{r_3},L_{l_4,m_4}^{r_4},L_{l_5,m_5}^{r_5}]]\\
=&M(\alpha_1,\alpha_4, \alpha_5)[L_{l_1+l_4+l_5-1,m_1+m_4+m_5-1}^{r_1+r_4+r_5},L_{l_2,m_2}^{r_2},L_{l_3,m_3}^{r_3}]\\
+&M(\alpha_2,\alpha_4, \alpha_5)[L_{l_2+l_4+l_5-1,m_2+m_4+m_5-1}^{r_2+r_4+r_5},L_{l_3,m_3}^{r_3},L_{l_1,m_1}^{r_1}]\\
+&M(\alpha_3,\alpha_4, \alpha_5)[L_{l_3+l_4+l_5-1,m_3+m_4+m_5-1}^{r_3+r_4+r_5},L_{l_1,m_1}^{r_1},L_{l_2,m_2}^{r_2}]\\
=&M(\alpha_1,\alpha_4, \alpha_5)M(\alpha_1+\alpha_4+\alpha_5-\nu, \alpha_2,\alpha_3)L_{l_1+l_2+l_3+l_4+l_5-2,m_1+m_2+m_3+m_4+m_5-2}^{r_1+r_2+r_3+r_4+r_5}\\
+&M(\alpha_2,\alpha_4, \alpha_5)M(\alpha_2+\alpha_4+\alpha_5-\nu, \alpha_1,\alpha_3)L_{l_1+l_2+l_3+l_4+l_5-2,m_1+m_2+m_3+m_4+m_5-2}^{r_1+r_2+r_3+r_4+r_5}\\
+&M(\alpha_3,\alpha_4, \alpha_5)M(\alpha_3+\alpha_4+\alpha_5-\nu, \alpha_1,\alpha_2)L_{l_1+l_2+l_3+l_4+l_5-2,m_1+m_2+m_3+m_4+m_5-2}^{r_1+r_2+r_3+r_4+r_5}.\\
\end{split}
\end{equation*}
Thanks to the properties of determinant, \eqref{eq:Jaco} holds.
 Since
\begin{equation*}
\begin{split}
&[L_{l_1,m_1}^{r_1},L_{l_2,m_2}^{r_2},L_{l_3,m_3}^{r_3}\cdot L_{l_4,m_4}^{r_4}]=[L_{l_1,m_1}^{r_1},L_{l_2,m_2}^{r_2},L_{l_3+l_4,m_3+m_4}^{r_3+r_4}]\\
=&M(\alpha_1,\alpha_2, \alpha_3+\alpha_4)L_{l_1+l_2+l_3+l_4-1,m_1+m_2+m_3+m_4-1}^{r_1+r_2+r_3+r_4}\\
=&M(\alpha_1,\alpha_2, \alpha_3)L_{l_1+l_2+l_3+l_4-1,m_1+m_2+m_3+m_4-1}^{r_1+r_2+r_3+r_4}\\
+&M(\alpha_1,\alpha_2, \alpha_4)L_{l_1+l_2+l_3+l_4-1,m_1+m_2+m_3+m_4-1}^{r_1+r_2+r_3+r_4}\\
=&[L_{l_1,m_1}^{r_1},L_{l_2,m_2}^{r_2},L_{l_3,m_3}^{r_3}]\cdot L_{l_4,m_4}^{r_4}
+L_{l_3,m_3}^{r_3}\cdot[L_{l_1,m_1}^{r_1},L_{l_2,m_2}^{r_2}, L_{l_4,m_4}^{r_4}],
\end{split}
\end{equation*}
\eqref{eq:Lei1} holds. Therefore,  $(A,\cdot,[ , , ])$ is an unital 3-Lie Poisson algebra.
\end{proof}

 In the following,  {\bf $\mathfrak{L}$ denotes the unital 3-Lie Poisson algebra $ (A,\cdot,[,  , ])$ in Theorem \ref{thm:3-Jacobi}, and $\mathfrak{A}$ denotes  the $3$-Lie algebra structure of $\mathfrak{L}$.}

 Now we give some symbols.  Let
$$\mathfrak{A}_1~(0,~~1~~;  ~~1,~~0), \quad \mathfrak{A}_2~(0,~~1~~; ~~1,~~0),  \quad \mathfrak{A}_3~(0,~~1~~ ; ~~1,~~0), \quad \mathfrak{A}_4~(\frac{1}{2}) \quad \mbox{ and } \quad \mathfrak{A}_5~(\frac{1}{2})$$
 \\be 3-Lie subalgebras of 3-Lie algebra  $\mathfrak{A}$ generated by the subsets

$$S_1=\big\{-L_{0,1}^{r}+\sqrt{-1}rL_{1,0}^{r},~~ -L_{1,0}^{r} |~~r\in \mathbb{Z}\big\}, ~~ S_2=\big\{-L_{0,1}^{r}-\sqrt{-1}rL_{1,0}^{r},~~ L_{1,0}^{r} |~~ r\in \mathbb{Z}\big\},$$

 $$S_3=\big\{L_{0,1}^{r},  L_{1,0}^{-r} |~~r\in \mathbb Z\big\} ,~~  \quad S_4=\big\{ L_{\frac{1}{2},r}^{(-1)^{r}} |~~ r\in \mathbb Z\big\}, \quad S_5=\big\{  L_{\frac{1}{2},m+\frac{1}{2}}^{r} |~~ r, m\in \mathbb Z\big\},$$ respectively. Then
\begin{eqnarray}
\label{eq:u1}\mathfrak{A}_1~(0,~~1~~;  ~~1,~~0)&:~~=\langle S_1\rangle=&\langle~~\big\{-L_{0,1}^{r}+\sqrt{-1}rL_{1,0}^{r},~~ -L_{1,0}^{r} |~~r\in \mathbb{Z}\big\}~~\rangle,\\
\label{eq:u2}\mathfrak{A}_2~(0,~~1~~; ~~1,~~0)&:~~=\langle S_2\rangle=&\langle ~~\big\{-L_{0,1}^{r}-\sqrt{-1}rL_{1,0}^{r},~~ L_{1,0}^{r} |~~ r\in \mathbb{Z}\big\}~~\rangle,\\
\label{eq:u3}\mathfrak{A}_3~(0,~~1~~ ; ~~1,~~0)&:~~=\langle S_3\rangle=&\langle~~ \big\{L_{0,1}^{r},  L_{1,0}^{-r} |~~r\in \mathbb Z\big\}~~\rangle,\\
\label{eq:u4}\mathfrak{A}_4~(\frac{1}{2})&:~~=\langle S_4\rangle=&\langle ~~ \big\{ L_{\frac{1}{2},r}^{(-1)^{r}} |~~ r\in \mathbb Z\big\}~~\rangle,\\
\label{eq:u5} \mathfrak{A}_5~(\frac{1}{2})&:~~=\langle S_5\rangle=&\langle~~ \big\{  L_{\frac{1}{2},m+\frac{1}{2}}^{r} |~~ r, m\in \mathbb Z\big\}~~\rangle.
\end{eqnarray}

In \cite{Takh}, authors gave the canonical Nambu 3-Lie algebra 
on $\mathbb{R}$.
Let $V=C^{\infty}(\mathbb R^n)$ (or $V=C^{\infty}(\mathbb C^n)$ ). Then  $(V,\{ , \cdots,  \})$ is  a \textbf{canonical Nambu $n$-Lie algebra},  where  for  $\forall f_1,f_2,\cdots,f_n\in V$ with $n$ coordinates $x_1,x_2,\cdots,x_n\in \mathbb R$ ( or $x_1,x_2,\cdots,x_n\in \mathbb C$),
\begin{equation}\label{eq:n-Nambu bracket}
  \{f_1,f_2,\cdots,f_n\}=\frac{\partial(f_1,f_2,\cdots,f_n)}{\partial(x_1,x_2,\cdots,x_n)}.
\end{equation}

\begin{theorem} Let $\mathfrak{L}$ be the 3-Lie Poisson algebra in Theorem \ref{thm:3-Jacobi} over the real field $\mathbb R$ (or in the complex field $\mathbb C$), then $\mathfrak{L}$ can be realized by a canonical Nambu 3-Lie algebra.
\end{theorem}

\begin{proof} Let
\begin{equation}
X_{l,m}^{r}=y^lz^me^{rx},~~ ~~\forall~ r,l,m\in \mathbb{Z}+\frac{1}{2}\mathbb{Z},
\end{equation}
 where $y, z\in \mathbb R^+ =[0, +\infty], x\in \mathbb R$ (or $x, y, z\in \mathbb C$).
Then  $$B=\sum\limits_{l, m, r\in \mathbb{Z}+\frac{1}{2}\mathbb{Z}}\mathbb R X_{l, m}^r \quad\mbox{  (or} \quad B=\sum\limits_{l, m, r\in \mathbb{Z}+\frac{1}{2}\mathbb{Z}}\mathbb C X_{l, m}^r  )$$ is an associative commutative algebra with the multiplication
  \begin{equation*}
    X_{l_1,m_1}^{r_1}\cdot X_{l_2,m_2}^{r_2}=y^{l_1}z^{m_1}e^{r_1x}\cdot y^{l_2}z^{m_2}e^{r_2x}=y^{l_1+l_2}z^{m_1+m_2}e^{(r_1+r_2)x}=X_{l_1+l_2,m_1+m_2}^{r_1+r_2}.
  \end{equation*}

 Define the 3-Lie multiplication
\begin{equation}\label{eq:Poisson3}
[X_{l_1,m_1}^{r_1},X_{l_2,m_2}^{r_2},X_{l_3,m_3}^{r_3}]= \left|
  \begin{array}{ccc}
 \frac{\partial X_{l_1,m_1}^{r_1}}{\partial x} & \frac{\partial X_{l_2,m_2}^{r_2}}{\partial x} & \frac{\partial X_{l_3,m_3}^{r_3}}{\partial x} \\\\
 \frac{\partial X_{l_1,m_1}^{r_1}}{\partial y} & \frac{\partial X_{l_2,m_2}^{r_2}}{\partial y} & \frac{\partial X_{l_3,m_3}^{r_3}}{\partial y} \\\\
 \frac{\partial X_{l_1,m_1}^{r_1}}{\partial z} & \frac{\partial X_{l_2,m_2}^{r_2}}{\partial z} & \frac{\partial X_{l_3,m_3}^{r_3}}{\partial z} \\
\end{array}
\right|.
\end{equation}

Thanks to Eq \eqref{eq:Jacobian}, $(B, \cdot,[ , , ])$ is a  3-Lie Poisson algebra, and
\begin{equation*}
  \begin{split}
    &[X_{l_1,m_1}^{r_1},X_{l_2,m_2}^{r_2},X_{l_3,m_3}^{r_3}]=[y^{l_1}z^{m_1}e^{r_1x},y^{l_2}z^{m_2}e^{r_2x},y^{l_3}z^{m_3}e^{r_3x}]\\\\
   =&\left|
       \begin{array}{ccc}
         r_1y^{l_1}z^{m_1}e^{r_1x}   & r_2y^{l_2}z^{m_2}e^{r_2x}   & r_3y^{l_3}z^{m_3}e^{r_3x} \\
         l_1y^{l_1-1}z^{m_1}e^{r_1x} & l_2y^{l_2-1}z^{m_2}e^{r_2x} & l_3y^{l_3-1}z^{m_3}e^{r_3x} \\
         m_1y^{l_1}z^{m_1-1}e^{r_1x} & m_2y^{l_2}z^{m_2-1}e^{r_2x} & m_3y^{l_3}z^{m_3-1}e^{r_3x} \\
       \end{array}
     \right|\\
   =&M(\alpha_1,\alpha_2, \alpha_3)y^{l_1+l_2+l_3-1}z^{m_1+m_2+m_3-1}e^{r_1+r_2+r_3}\\
   =&M(\alpha_1,\alpha_2, \alpha_3)X_{l_1+l_2+l_3-1,m_1+m_2+m_3-1}^{r_1+r_2+r_3},\\
  \end{split}
  \end{equation*}
where $\alpha_i=(l_i, m_i, r_i)\in (\mathbb{Z}+\frac{1}{2}\mathbb{Z})^{\otimes 3}, 1\leq i\leq 3$. Therefore, $\chi: \mathfrak{L}\rightarrow B$ defined by, $\chi(L_{l, m}^r)=X_{l, m}^r$, for all $l, m, r\in \mathbb{Z}+\frac{1}{2}\mathbb{Z} $, is an algebra isomorphism between
$\mathfrak{L}$  and $(B, \cdot, [, , ])$.
  \end{proof}

\section{Applications of unital 3-Lie Poisson algebra  $\mathfrak{L}$}
\setcounter{equation}{0}
\renewcommand{\theequation}
{4.\arabic{equation}}

 In this section, we study some applications of $\mathfrak{L}$.  We will prove that four  important 3-Lie algebras:  3-Virasoro-Witt algebra $\mathcal{W}_3$ in \cite{CTDFC}, $A_\omega^\delta$ in \cite{BR-3},  $A_{\omega}$ in \cite{BR-2} and 3-$W_{\infty}$ algebra in \cite{CSAS}
  can be embedded in $\mathfrak{L}$.
\subsection{3-Virasoro-Witt algebra}

 Witt algebra (centerless Virasoro algebra) is an important  complex Lie algebra in two-dimensional conformal field theory and string theory. It has been generalized to higher arties as 3-Virasoro-Witt algebra in \cite{CTDFC}. First  we  show that how to construct  3-Virasoro-Witt algebra by creation and annihilation operators $a^{\dagger}$ and $a$.

We know that $a^{\dagger}$ and $a$ have wide applications in quantum mechanics. Consider the Schr\"{o}dinger equation for the one-dimensional time independent quantum harmonic oscillator
\begin{equation}
\label{eq:digereq} \left(-\frac{\hbar^2}{2m}\frac{d^2}{dx^2}+\frac{1}{2}m\omega^2x^2\right)\psi(x)=E\psi(x).
\end{equation}
Set
$$x=\sqrt{\frac{\hbar}{m\omega}}q,$$
 a direct computation yields that
$$\frac{\hbar\omega}{2}\left(-\frac{d^2}{dq^2}+q^2\right)\psi(q)=E\psi(q),
$$
moreover
\begin{equation}\label{eq:sding}
-{\frac {d^{2}}{dq^{2}}}+q^{2}=\left(-{\frac {d}{dq}}+q\right)\left({\frac {d}{dq}}+q\right)+{\frac {d}{dq}}q-q{\frac {d}{dq}}.
\end{equation}

For arbitrary differentiable function $f(q)$, since
$$\left({\frac {d}{dq}}q-q{\frac {d}{dq}}\right)f(q)={\frac {d}{dq}}(qf(q))-q{\frac {df(q)}{dq}}=f(q),$$
we have
\begin{equation}\label{eq:aa=1}
{\frac {d}{dq}}q-q{\frac {d}{dq}}=1.
\end{equation}
Therefore,  \eqref{eq:sding} can be reduced to
$$-{\frac {d^{2}}{dq^{2}}}+q^{2}=\left(-{\frac {d}{dq}}+q\right)\left({\frac {d}{dq}}+q\right)+1.$$
So the Schr\"{o}dinger equation \eqref{eq:digereq}  becomes
$$\hbar \omega \left[{\frac {1}{\sqrt {2}}}\left(-{\frac {d}{dq}}+q\right){\frac {1}{\sqrt {2}}}\left({\frac {d}{dq}}+q\right)+{\frac {1}{2}}\right]\psi (q)=E\psi (q).$$
Define the creation operator and annihilation operator as
$$a^{\dagger }\ =\ {\frac {1}{\sqrt {2}}}\left(-{\frac {d}{dq}}+q\right),~a ={\frac {1}{\sqrt {2}}}\left({\frac {d}{dq}}+q\right),$$
then the Schr\"{o}dinger equation reduces to
$$\hbar \omega \left(a^{\dagger }a+{\frac {1}{2}}\right)\psi (q)=E\psi (q).$$
It is a simplification of the Schr\"{o}dinger equation.
Furthermore,  set $p=-i{\frac {d}{dq}}$, then  $[q,p]=i$, and
 \begin{equation}\label{eq:a}
 a={\frac {1}{\sqrt {2}}}(q+ip),~~
a^{\dagger }={\frac {1}{\sqrt {2}}}(q-ip).
\end{equation}
Thanks to  \eqref{eq:aa=1} and \eqref{eq:a},
\begin{equation}\label{eq:a+a}
[a,a^{\dagger }]={\frac {1}{2}}[q+ip,q-ip]={\frac {1}{2}}([q,-ip]+[ip,q])={\frac {-i}{2}}([q,p]+[q,p])=1.\end{equation}
Set
\begin{equation}N=a^\dag a,~~~ L_n=-(a^\dag)^n(N+\gamma+n\beta),~~~M_n=(a^\dag)^n,\end{equation}
where $\gamma,\beta$ are parameters. Using the Nambu commutator
\begin{equation}\label{eq:Nambu-commutator}
[A,B,C]=A[B,C]+B[C,A]+C[A,B]=ABC-BAC+CAB-ACB+BCA-CBA,
\end{equation}
and Eq \eqref{eq:a+a}, we have
\begin{equation}\label{eq:bracket}
\left\{
\begin{split}
  &[L_n,L_m,L_k] = \beta(1-\beta)(n-m)(m-k)(n-k)M_{n+m+k}, \\
  &[M_n,L_m,L_k] = (m-k)(L_{n+m+k}+(1-2\beta)nM_{n+m+k}), \\
  &[M_n,M_m,L_k] = (m-n)M_{n+m+k}. \\
  &[M_n,M_m,M_k] = 0.
\end{split}\right.
\end{equation}

Replacing the basis in Eq \eqref{eq:bracket} with
$$Q_n=\frac{1}{\sqrt[4]{\beta(1-\beta)}}L_n,~~~ R_n=\sqrt[4]{\beta(1-\beta)}M_n,$$
 and taking  limit $\beta\rightarrow \infty$, we get the 3-Virasoro-Witt algebra.


\begin{lemma}\cite{CTDFC}\label{defn:3-Witt}
  Let $V$ be a commutative associative algebra with a basis $\{Q_{n},R_{n}\}_{n\in \mathbb{Z}}$, and define 3-ary linear skew-symmetric multiplication ~$[,,]$~on $V$ as follows
  \begin{equation}\label{3-Virasoro-Witt Algebra}
    \begin{cases}
[Q_{k}, Q_{m}, Q_{n}]=(k-m)(m-n)(k-n)R_{k+m+n}, \\
[Q_{p}, Q_{q}, R_{k}]=(p-q)(Q_{k+p+q}+zkR_{k+p+q}), \\
[Q_{p}, R_{q}, R_{k}]=(k-q)R_{k+p+q}, \\
[R_{p}, R_{q}, R_{k}]=0,
\end{cases} \forall~ k, m, n,p,q \in \mathbb{Z},
  \end{equation}
where z is a parameter. Then the multiplication \eqref{3-Virasoro-Witt Algebra} does not satisfy Eq \eqref{eq:Jaco}, except when $z=\pm2\sqrt{-1}$. In that cases, $(V, [,,]) $ is a 3-Lie algebra, which is called \textbf{3-Virasoro-Witt algebra},~and is denoted by $\mathcal{W}_3$.
\end{lemma}

\begin{theorem}\label{W3}  3-Virasoro-Witt algebras $\mathcal{W}_3({z=2\sqrt{-1}})$  and  $\mathcal{W}_3({z=-2\sqrt{-1}})$)  can be embedded in  $\mathfrak{L}$, and which are isomorphic to $\mathfrak{A}_1~(0,1~~;~~1,0)$ and $\mathfrak{A}_2~(0,1~~;1,0~~)$, respectively.
\end{theorem}

\begin{proof}
From \eqref{eq:3-Jacobi bracket}  and \eqref{eq:u1}, we have
\begin{equation*}
\begin{split}
  &[-L_{0,1}^{k}+\sqrt{-1}kL_{1,0}^{k},~~-L_{0,1}^{m}+\sqrt{-1}mL_{1,0}^{m},~~-L_{0,1}^{n}+\sqrt{-1}nL_{1,0}^{n}]\\
  =&[L_{0,1}^{k},~~L_{0,1}^{m},~~\sqrt{-1}nL_{1,0}^{n}]+[L_{0,1}^{k},~~\sqrt{-1}mL_{1,0}^{m},~~L_{0,1}^{n}]-[L_{0,1}^{k},~~\sqrt{-1}mL_{1,0}^{m},~~\sqrt{-1}nL_{1,0}^{n}]\\
  +&[\sqrt{-1}kL_{1,0}^{k},~~L_{0,1}^{m},~~L_{0,1}^{n}]-[\sqrt{-1}kL_{1,0}^{k},~~L_{0,1}^{m},~~\sqrt{-1}nL_{1,0}^{n}]-[\sqrt{-1}kL_{1,0}^{k},~~\sqrt{-1}mL_{1,0}^{m},~~L_{0,1}^{n}]\\
  =&\sqrt{-1}n\left|
                \begin{array}{ccc}
                  k & m & n \\
                  0 & 0 & 1 \\
                  1 & 1 & 0 \\
                \end{array}
              \right|L_{0,1}^{k+m+n}+\sqrt{-1}m\left|
                \begin{array}{ccc}
                  k & m & n \\
                  0 & 1 & 0 \\
                  1 & 0 & 1 \\
                \end{array}
              \right|L_{0,1}^{k+m+n}+mn\left|
                \begin{array}{ccc}
                  k & m & n \\
                  0 & 1 & 1 \\
                  1 & 0 & 0 \\
                \end{array}
              \right|L_{1,0}^{k+m+n}\\
  +&\sqrt{-1}k\left|
                \begin{array}{ccc}
                  k & m & n \\
                  1 & 0 & 0 \\
                  0 & 1 & 1 \\
                \end{array}
              \right|L_{0,1}^{k+m+n}+kn\left|
                \begin{array}{ccc}
                  k & m & n \\
                  1 & 0 & 1 \\
                  0 & 1 & 0 \\
                \end{array}
              \right|L_{1,0}^{k+m+n}+km\left|
                \begin{array}{ccc}
                  k & m & n \\
                  1 & 1 & 0 \\
                  0 & 0 & 1 \\
                \end{array}
              \right|L_{1,0}^{k+m+n}\\
  =&\sqrt{-1}(n(m-k)+m(k-n)+k(n-m))L_{1,0}^{k+m+n}\\
  +&(mn(m-n)+kn(n-k)+km(k-m))L_{1,0}^{k+m+n}\\
  =&(k-m)(m-n)(k-n)L_{1,0}^{k+m+n};\\
  &[-L_{0,1}^{p}+\sqrt{-1}pL_{1,0}^{p},~~-L_{0,1}^{p}+\sqrt{-1}pL_{1,0}^{p},~~-L_{1,0}^{k}]\\
  =&-[L_{0,1}^{p},~~L_{0,1}^{q},~~L_{1,0}^{k}]+[L_{0,1}^{p},~~\sqrt{-1}qL_{1,0}^{q},~~L_{1,0}^{k}]+[\sqrt{-1}pL_{1,0}^{p},~~L_{0,1}^{q},~~L_{1,0}^{k}]\\
  =&-\left|
      \begin{array}{ccc}
        p & q & k \\
        0 & 0 & 1 \\
        1 & 1 & 0 \\
      \end{array}
    \right|L_{0,1}^{p+q+k}+q\sqrt{-1}\left|
      \begin{array}{ccc}
        p & q & k \\
        0 & 1 & 1 \\
        1 & 0 & 0 \\
      \end{array}
    \right|L_{1,0}^{p+q+k}+p\sqrt{-1}\left|
      \begin{array}{ccc}
        p & q & k \\
        1 & 0 & 1 \\
        0 & 1 & 0 \\
      \end{array}
    \right|L_{1,0}^{p+q+k}\\
  =&(p-q)L_{0,1}^{p+q+k}+\sqrt{-1}q(q-k)L_{1,0}^{p+q+k}+\sqrt{-1}p(k-p)L_{1,0}^{p+q+k}\\
  =&(p-q)L_{0,1}^{p+q+k}+\sqrt{-1}(q^2+qk-pk-p^2-pq+pq-2qk+2pk)L_{1,0}^{p+q+k}\\
  =&(p-q)((L_{0,1}^{p+q+k}-\sqrt{-1}(p+q+k)L_{1,0}^{p+q+k})+2\sqrt{-1}kL_{1,0}^{p+q+k});
\end{split}
\end{equation*}

\begin{equation*}
\begin{split}
  &[-L_{0,1}^{p}+\sqrt{-1}pL_{1,0}^{p},-L_{1,0}^{q},-L_{1,0}^{k}]=-[L_{0,1}^{p},L_{1,0}^{q},L_{1,0}^{k}]
  =-\left|
      \begin{array}{ccc}
        p & q & k \\
        0 & 1 & 1 \\
        1 & 0 & 0 \\
      \end{array}
    \right|L_{1,0}^{p+q+k}\\
  =&(k-q)L_{1,0}^{p+q+k};\\
  &[-L_{1,0}^{p},~~-L_{1,0}^{q},~~-L_{1,0}^{k}]=0;~~\forall~ k,m,n,p,q\in \mathbb{Z}.\\
\end{split}
\end{equation*}
Thanks to \eqref{3-Virasoro-Witt Algebra}, $\Gamma_1:$  $\mathfrak{A}_1(0,1~~;~~1,0)\rightarrow$ $\mathcal{W}_3(z={2\sqrt{-1}})$ defined by
$$\Gamma_1(-L_{0,1}^{r}+\sqrt{-1}rL_{1,0}^{r})=Q_r, ~~ \Gamma_1(-L_{1,0}^{r})=R_r, ~~ \forall r\in\mathbb Z,$$
 is a 3-Lie algebra isomorphism.  By a similar discussion to the above,   $\Gamma_2:$  $\mathfrak{A}_2(0,1~~;~~1,0)\rightarrow$ $\mathcal{W}_3(z={-2\sqrt{-1}})$
defined as the above is a 3-Lie algebra isomorphism. Therefore, 3-Virasoro-Witt algebras $\mathcal{W}_3(z={2\sqrt{-1}})$  and $\mathcal{W}_3(z={-2\sqrt{-1}})$  can be embedded in  $\mathfrak{L}$.
\end{proof}

\subsection{The 3-Lie algebra $A_\omega^\delta$ constructed  by an involution and a derivation}

In this subsection, we study an infinite-dimensional 3-Lie algebra $A_\omega^\delta$ which is constructed by a
commutative associative algebra,  an involution and a derivation (\cite{BR-3}).

\begin{lemma}\cite{BR-3}\label{defn:Aw2}
 Let $E$ be a commutative associative algebra with a basis $\{S_{r},T_{r}\}_{r\in \mathbb{Z}}.$ Then  $A_\omega^\delta:=(E,[,,])$ is a  simple 3-Lie algebra, where
\begin{equation}\label{eq:Aw2}
\begin{cases}
[S_{l}, S_{m}, S_{n}]=0, \\
[T_{l}, T_{m}, T_{n}]=0, \\
[S_{l}, S_{m}, T_{n}]=(m-l)S_{l+m-n}, \\
[S_{l}, T_{m}, T_{n}]=(n-m)T_{m+n-l},
\end{cases} \forall~ l, m, n \in \mathbb{Z}.
\end{equation}
\end{lemma}

\begin{theorem}\label{Aw}
  The linear mapping $\rho:A_\omega^{\delta}\rightarrow \mathfrak{A}_3~(0,1~~;~~1,0)$  defined by
 $$\rho(S_r)= L_{0,1}^{r},~~\rho(T_r)= L_{1,0}^{-r}, \quad \forall r\in \mathbb Z, $$
 is a 3-Lie algebra isomorphism, therefore, 3-Lie  algebra $A_\omega^{\delta}$ can be embedded in  $\mathfrak{L}$.
\end{theorem}
\begin{proof}
By Eqs~\eqref{eq:3-Jacobi bracket} and \eqref{eq:u3}, and a direct computation, we have
\begin{equation*}
\begin{split}
  &[\rho(S_{l}), ~~\rho(S_{m}),~~ \rho(T_{n})]=[L_{0,1}^{l},~~L_{0,1}^{m},~~L_{1,0}^{-n}]=\left|
                                                    \begin{array}{ccc}
                                                      l & m & -n \\
                                                      0 & 0 & 1 \\
                                                      1 & 1 & 0 \\
                                                    \end{array}
                                                  \right|L_{0,1}^{l+m-n}=(m-l)L_{0,1}^{l+m-n},\\
  &[\rho(S_{l}), \rho(T_{m}), \rho(T_{n})]=[L_{0,1}^{l},~~L_{1,0}^{-m},~~L_{1,0}^{-n}]=\left|
                                                    \begin{array}{ccc}
                                                      l & -m & -n \\
                                                      0 & 1 & 1 \\
                                                      1 & 0 & 0 \\
                                                    \end{array}
                                                  \right|L_{1,0}^{l-m-n}=(n-m)L_{1,0}^{l-m-n},\\
  &[\rho(T_{l}), \rho(T_{m}), \rho(T_{n})]=[L_{0,1}^{l},~~L_{0,1}^{m},~~L_{0,1}^{n}]=0,\\
  &[\rho(S_{l}), \rho(S_{m}), \rho(T_{n})]=[L_{1,0}^{-l},~~L_{1,0}^{-m},~~L_{1,0}^{-n}]=0, ~~\forall~ l,m,n\in \mathbb{Z}.
\end{split}
\end{equation*}
Thanks to  Lemma \ref{defn:Aw2},
 \begin{equation*}
   \begin{split}
&\rho([S_{l}, S_{m}, S_{n}])=0, ~~\rho([T_{l}, T_{m}, T_{n}])=0, \\
&\rho([S_{l}, S_{m}, T_{n}])=(m-l)\rho(S_{l+m-n})=(m-l)L_{1,0}^{l+m-n}, \\
&\rho([S_{l}, T_{m}, T_{n}])=(n-m)\rho(T_{m+n-l})=(n-m)L_{0,1}^{l-m+n},~~~\forall~ l, m, n \in \mathbb{Z}.
    \end{split}
 \end{equation*}
\newpage
Then  we have
 \begin{equation*}
   \begin{split}
&\rho([S_{l}, S_{m}, S_{n}])=[\rho(S_{l}), \rho(S_{m}), \rho(S_{n}]), ~~\rho([T_{l}, T_{m}, T_{n}])=[\rho(T_{l}), \rho(T_{m}), \rho(T_{n}]), \\
&\rho([S_{l}, S_{m}, T_{n}])=[\rho(S_{l}), \rho(S_{m}), \rho(T_{n}]), \rho([S_{l}, T_{m}, T_{n}])=[\rho(S_{l}), \rho(T_{m}), \rho(T_{n}]).
    \end{split}
 \end{equation*}
 Therefore,  3-Lie algebra $A_\omega^{\delta}$  can be embedded in $\mathfrak{L}$.
\end{proof}

\subsection{3-Lie algebra $A_{\omega}$ constructed by Laurent polynomials}

Now we study the  infinite dimensional 3-Lie algebra $A_{\omega}$ which is constructed by Laurent polynomials in \cite{BR-2}

\begin{lemma}\cite{BR-2}\label{defn:Aw}
  Let U be a vector space with a basis $\{U_n \}_{n\in \mathbb{Z}}$ over $\mathbb F$. Then $U$ is a simple 3-Lie algebra with  the multiplication
\begin{equation}\label{eq:Aw1}
  [U_l,U_m,U_n]=\left|
                  \begin{array}{ccc}
                    (-1)^l & (-1)^m & (-1)^n \\
                    1      & 1      & 1 \\
                    l      & m      & n \\
                  \end{array}
                \right|U_{l+m+n-1},\forall~ l,m,n\in\mathbb{Z}.
\end{equation}
The 3-Lie algebra $(U,[,,])$ is denoted by $A_{\omega}$.
\end{lemma}

\begin{theorem}
The linear mapping $\tau:A_{\omega}\rightarrow \mathfrak{A}_4~(\frac{1}{2}),$ defined by
$$\tau(U_{r})= \sqrt{2}L_{\frac{1}{2},r}^{(-1)^{r}}, ~~ \forall r\in \mathbb{Z},$$ is a 3-Lie algebra isomorphism, therefore, the 3-Lie algebra $A_{\omega}$ can be embedded in  $\mathfrak{L}$.
\end{theorem}

\begin{proof}
Thanks to Eqs~\eqref{eq:3-Jacobi bracket} and \eqref{eq:u3},
\begin{equation*}
\begin{split}
  &[\tau(U_l),~~\tau(U_m),~~\tau(U_n)] =[\sqrt{2}L_{\frac{1}{2},l}^{(-1)^{l}},~~\sqrt{2}L_{\frac{1}{2},m}^{(-1)^{m}},~~\sqrt{2}L_{\frac{1}{2},n}^{(-1)^{n}}]\\
  =&2\sqrt{2}\left|
      \begin{array}{ccc}
        (-1)^l      & (-1)^m      & (-1)^n \\
        \frac{1}{2} & \frac{1}{2} & \frac{1}{2} \\
        l           & m           & n \\
      \end{array}
    \right|L_{\frac{1}{2},l+m+n-1}^{(-1)^l+(-1)^m+(-1)^n}\\
  =&\left|
      \begin{array}{ccc}
        (-1)^l  & (-1)^m & (-1)^n \\
        1       & 1      & 1 \\
        l       & m      & n \\
      \end{array}
    \right|\sqrt{2}L_{\frac{1}{2},l+m+n-1}^{(-1)^l+(-1)^m+(-1)^n}\\
    =&\left|
      \begin{array}{ccc}
        (-1)^l  & (-1)^m & (-1)^n \\
        1       & 1      & 1 \\
        l       & m      & n \\
      \end{array}
    \right|\sqrt{2}L_{\frac{1}{2},l+m+n-1}^{(-1)^{l+m+n-1}},~~\forall~ l,m,n\in \mathbb{Z}.\\
\end{split}
\end{equation*}
By Lemma \ref{defn:Aw},
\begin{equation*}
\begin{split}
  \tau([U_l,~~U_m,~~U_n])&=\left|
                  \begin{array}{ccc}
                    (-1)^l & (-1)^m & (-1)^n \\
                    1      & 1      & 1 \\
                    l      & m      & n \\
                  \end{array}
                \right|\tau(U_{l+m+n-1})\\
   &=\left|
      \begin{array}{ccc}
        (-1)^l  & (-1)^m & (-1)^n \\
        1       & 1      & 1 \\
        l       & m      & n \\
      \end{array}
    \right|\sqrt{2}L_{\frac{1}{2},l+m+n-1}^{(-1)^{l+m+n-1}}.
\end{split}
\end{equation*}
It follows
\begin{equation*}
  \tau([U_l,~~U_m,~~U_n])=[\tau(U_l),~~\tau(U_m), ~~\tau(U_n)],
\end{equation*}
 $\tau$ is a 3-Lie algebra isomorphism, therefore, the 3-Lie algebra $A_{\omega}$ can be embedded in  $\mathfrak{L}$.
\end{proof}

\vspace{2mm}\subsection{3-$W_{\infty}$ algebra}
The algebra  $W_{\infty}$ is a higher-spin extension of the Virasoro algebra (\cite{Bakas}). In \cite{CSAS},  authors  obtained a 3-$W_{\infty}$ algebra by using ``lone-star" product and commutative
relations of generators in  $W_{\infty}$ and  appropriate double scaling limits on the generators.
\begin{lemma}\cite{CSAS}\label{defn:3-W00}
   Let $W$ be a commutative associative algebra with a basis $\{W_{m}^{r}\}_{r,m\in \mathbb{Z}}.$ Then  $(W,[,,])$ is a 3-Lie algebra with the multiplication
   \begin{equation}\label{3-Woo}
     [W_{m_1}^{r_1},W_{m_2}^{r_2},W_{m_3}^{r_3}]=\left|
                                                   \begin{array}{ccc}
                                                     1 & 1 & 1 \\
                                                     m_1 & m_2 & m_3 \\
                                                     r_1 & r_2 & r_3 \\
                                                   \end{array}
                                                 \right|W_{m_1+m_2+m_3+1}^{r_1+r_2+r_3},\forall~ r_1,r_2,r_3,m_1,m_2,m_3\in\mathbb{Z},
   \end{equation}
   which is called \textbf{3-$W_{\infty}$ algebra}, and is simply denoted by 3-$W_{\infty}$.
\end{lemma}

\begin{theorem}
The linear mapping $\phi:3\textrm{-}W_{\infty}\rightarrow \mathfrak{A}_5~(\frac{1}{2})$ defined by
$$\phi(W_{m}^{r})= \sqrt{2}L_{\frac{1}{2},m+\frac{1}{2}}^{r}, \quad \forall m,r\in \mathbb{Z},$$
is a 3-Lie algebra isomorphism. Therefore, 3-$W_{\infty}$ can be embedded in  $\mathfrak{L}$.
\end{theorem}

\begin{proof}
From  ~\eqref{eq:3-Jacobi bracket} and \eqref{eq:u5}, and Lemma.\ref{defn:3-W00}, for $\forall m_1,m_2,m_3,r_1,r_2,r_3\in \mathbb{Z},$
\begin{equation*}
\begin{split}
 & [\phi(W_{m_1}^{r_1}), ~~\phi(W_{m_2}^{r_2}), ~~\phi(W_{m_3}^{r_3})]  =[L_{\frac{1}{2},m_1+\frac{1}{2}}^{r_1},~~L_{\frac{1}{2},m_2+\frac{1}{2}}^{r_2}, ~~L_{\frac{1}{2},m_3+\frac{1}{2}}^{r_3}]\\
   =&2\sqrt{2}\left|
      \begin{array}{ccc}
        r_1 & r_2 &  r_3\\
        \frac{1}{2} & \frac{1}{2} & \frac{1}{2} \\
        m_1+\frac{1}{2} & m_2+\frac{1}{2} & m_3+\frac{1}{2} \\
      \end{array}
    \right|L_{\frac{1}{2},m_1+m_2+m_3+\frac{1}{2}}^{r_1+r_2+r_3}\\
  =&\left|
       \begin{array}{ccc}
          1 & 1 & 1 \\
         m_1 & m_2 & m_3 \\
          r_1 & r_2 & r_3 \\
         \end{array}
       \right|\sqrt{2}L_{\frac{1}{2},m_1+m_2+m_3+\frac{1}{2}}^{r_1+r_2+r_3}\\
       =&\left|
                                                   \begin{array}{ccc}
                                                     1 & 1 & 1 \\
                                                     m_1 & m_2 & m_3 \\
                                                     r_1 & r_2 & r_3 \\
                                                   \end{array}
                                                 \right|\phi(W_{m_1+m_2+m_3+1}^{r_1+r_2+r_3})
                                                  =\phi([W_{m_1}^{r_1},W_{m_2}^{r_2},W_{m_3}^{r_3}]),
\end{split}
\end{equation*}
 therefore, the 3-$W_{\infty}$ algebra is isomorphic to $\mathfrak{A}_5~(\frac{1}{2})$ and 3-$W_{\infty}$ algebra can be embedded in  $\mathfrak{L}$.
\end{proof}

\section{Structure of unital 3-Lie Poisson algebra  $\mathfrak{L}$}
\setcounter{equation}{0}
\renewcommand{\theequation}
{5.\arabic{equation}}

 From \eqref{eq:unit} and \eqref{eq:3-Jacobi bracket},  $L_{0,0}^{0}$ is the center of $\mathfrak{L}$, that is, $[L_{0,0}^{0}, \mathfrak{L}, \mathfrak{L} ]=0$ and  $L_{0,0}^{0}\cdot L_{l, m}^r=L_{l, m}^r\cdot L_{0, 0}^0$ for all $l, m, r\in \mathbb Z+\frac{1}{2}\mathbb Z$. Denotes $C_0=\mathbb F L_{0,0}^{0}$. ~Thanks to \eqref{eq:unit}, and $[L_{0,0}^{0},\mathfrak{A},\mathfrak{A}]=0$,~~$C_0$~is an ideal of~$\mathfrak{L}$.

\begin{theorem} The quotient 3-Lie algebra
$\mathfrak{A}~/C_0$ is an infinite dimensional simple 3-Lie algebra.
\end{theorem}

\begin{proof}
Let $I$ be a nonzero ideal of the quotient algebra $\mathfrak{A}~/C_{0}$, and $ u\in I_{\neq C_0}$. Then we can suppose
$$u=\sum_{i=1}^{n}a_iL_{l_{i},m_{i}}^{r_i}+C_0,~~r_i,l_i,m_i\in \mathbb{Z}+\frac{1}{2}\mathbb{Z},~~ a_i\in \mathbb F_{\neq 0}, 1\leq i\leq n.$$

If $n=1$, then $a_1\neq 0$, and for all $L_{l, m}^r+C_0\neq C_0$ satisfying  $(l+1, m+1, r)\neq k(l_1, m_1, r_1)$ (  $\forall k\in  \mathbb{Z}+\frac{1}{2}\mathbb{Z}$),  there exist
$a,s, t \in \mathbb{Z}+\frac{1}{2}\mathbb{Z}$, such that $ \left|\begin{array}{ccc}
r_1 & a & r \\
l_1 & s & l+1 \\
m_1 &t & m+1 \\
\end{array}
\right|\neq 0.$
Then \begin{equation*}
\begin{split}
  &[u,L_{l_1+s,m_1+t}^{r_1+a}+C_0, L_{l-2l_1-s+1,m-2m_1-t+1}^{r-2r_1-a}+C_0]\\
  =&a_1[L_{l_1,m_1}^{r_1}, L_{l_1+s,m_1+t}^{r_1+a}, L_{l-2l_1-s+1,m-2m_1-t+1}^{r-2r_1-a}]+C_0\\
  =&a_1\left|\begin{array}{ccc}
  r_1 & r_1+a & r-2r_1-a \\
  l_1 & l_1+s & l-2l_1-s+1 \\
  m_1&m_1+t & m-2m_1-t+1 \\
  \end{array}
  \right|L_{l,m}^{r}+C_0\\
   =&a_1 \left|\begin{array}{ccc}
    r_1 & a & r \\
    l_1 & s & l+1 \\
     m_1 &t & m+1 \\
     \end{array}
     \right|L_{l,m}^{r}+C_0\in I_{\neq 0}, \\
\end{split}
\end{equation*}
therefore,   for all $L_{l, m}^r+C_0\neq C_0$ satisfying  $(l+1, m+1, r)\neq k(l_1, m_1, r_1)$, $k\in  \mathbb{Z}+\frac{1}{2}\mathbb{Z}$, we have $L_{l,m}^r+C_0\in I$.

For $L_{l, m}^r+C_0\neq C_0$ satisfying  $(l+1, m+1, r)=k(l_1, m_1, r_1)$, $k\in  \mathbb{Z}+\frac{1}{2}\mathbb{Z}$, there are $l', m', r'\in  \mathbb{Z}+\frac{1}{2}\mathbb{Z}$, such that
  $(l+1, m+1, r)\neq k(l', m', r')$ for all $k\in  \mathbb{Z}+\frac{1}{2}\mathbb{Z}$. By the above discussion, we have $L_{l', m'}^{r'}+C_0\in I_{\neq C_0}$, therefore, $L_{l,m}^{r}+C_0\in I.$ It follows $I=$$\mathfrak{A}~/C_0$.

Now assume $n=t\in \mathbb Z^+$ is true. We will discuss the case $n=t+1$. Thanks to  \eqref {eq:3-Jacobi bracket}, there are $ r_0\in \mathbb{Z}+\frac{1}{2}\mathbb{Z}$ such that
$r_0\neq 0$ and $r_i+r_{t+1}+r_0\neq 0$ for all $1\leq i\leq t$. Then
\begin{equation*}
  \begin{split}
  &[u, L_{l_{t+1},m_{t+1}}^{r_{t+1}}+C_0, L_{-l_{t+1}+1,-m_{t+1}+1}^{r_0}+C_0]\\
  =& [\sum_{i=1}^{t+1}a_iL_{l_i,m_i}^{r_i}, L_{l_{t+1},m_{t+1}}^{r_{t+1}}, L_{-l_{t+1}+1,-m_{t+1}+1}^{r_0}]+C_0\\
  =&\sum_{i=1}^{t}a_iL_{l_i,m_i}^{r_i+r_0+r_{t+1}}+C_0\in I_{\neq C_0}.
  \end{split}
\end{equation*}
Therefore, we have $I=$$\mathfrak{A}~/C_0$. The proof is complete.
\end{proof}

\begin{theorem} 1) The  subspace  $H$ spanned by
\begin{equation}\label{eq:cardon}
  \{L_{i,i}^{t} \mid ~~ t,i\in \mathbb{Z}+\frac{1}{2}\mathbb{Z}\}
\end{equation}  is a subalgebra of $\mathfrak{L}$ and satisfies  $[H, H, H]=0$. Furthermore,
for any $L^r_{l, m}\in \mathfrak{L}$, $[L^r_{l, m}, H, H]\subseteq H$ if and only if $L^r_{l, m}\in H$, and $L_{l,m}^r\cdot L_{i,i}^t\in H$  if and only if $L^r_{l, m}\in H$, that is, $l=m$.

2) For any $i, t\in \mathbb{Z}+\frac{1}{2}\mathbb{Z}$, the left multiplication
$\mbox{ad}(L_{i,i}^{t}, L_{-i+1,-i+1}^{-t}): \mathfrak{L}\rightarrow  \mathfrak{L} $ is semi-simple, that is,
$$ \mathfrak{L}=\sum\limits_{k\in \mathbb{Z}+\frac{1}{2}\mathbb{Z}}\mathfrak{L}(k, t),$$
where $\mathfrak{L}(k, t)$ is the subspace of $\mathfrak{L}$ spanned by
\begin{equation*}
\begin{split}
&\big\{ L^r_{l,m} |~~ \mbox{ad}(L_{i,i}^{t}, L_{-i+1,-i+1}^{-t})L_{l,m}^r=k L_{l, m}^r, ~~l, m, r\in \mathbb{Z}+\frac{1}{2}\mathbb{Z} \big\}\\
=&\{ L_{l, m}^r |~~ l, m, r\in  \mathbb{Z}+\frac{1}{2}\mathbb{Z}, \mbox{and}~~ t(m-l)=k\}.
\end{split}
\end{equation*}

Furthermore, if $t\neq 0$, then $H=\mathfrak{L}(0, t)$. If $t=0$, then $\mathfrak{L}(0, 0)=\mathfrak{L}$.
\end{theorem}

\begin{proof} Thanks to \eqref{eq:3-Jacobi bracket}, for all $L_{i,i}^{s}, L_{j,j}^{t},  L_{h,h}^{r} \in H$
$$L_{i,i}^{s}\cdot L_{j,j}^{t}=L_{i+j,i+j}^{s+t}\in H, ~~~~ [L_{i,i}^{s}, L_{j,j}^{t},  L_{h,h}^{r}]=0,$$
therefore, $H$ is a subalgebra of $\mathfrak{L}$ and $[H, H, H]=0.$

For any $L_{l,m}^{r}\in \mathfrak{L}$,  and $\forall i, j, s, t\in \mathbb{Z}+\frac{1}{2}\mathbb{Z}  $, by \eqref{eq:3-Jacobi bracket},

$[L_{l,m}^{r}, L_{i,i}^{s}, L_{j,j}^{t}]=\left| \begin{array}{ccc}
                                                                            r & s & t \\
                                                                            l & i & j \\
                                                                            m & i & j \\
                                                                          \end{array}
                                                                        \right|L_{i+j+l-1,i+j+m-1}^{r+s+t} \in H$
if and only if $i+j+l-1=i+j+m-1$, that is, $L^r_{l, m}\in H$. We get 1).

Apply \eqref{eq:3-Jacobi bracket} and a direct computation, for all $L_{l,m}^{r}\in \mathfrak{L}, $
\begin{equation}\label{eq:iit}\mbox{ad}(L_{i,i}^{t}, L_{-i+1,-i+1}^{-t})L_{l,m}^{r}=\left| \begin{array}{ccc}
                                                                             t & -t &r\\
                                                                             i & -i+1&l \\
                                                                             i & -i+1&m \\
                                                                          \end{array}
                                                                        \right|L_{l,m}^{r}=t(m-l) L_{l,m}^{r}.
\end{equation}
Therefore,  $L_{l,m}^{r} \in \mathfrak{L}(t(m-l), t)$, and
 $ \mathfrak{L}=\sum\limits_{k\in \mathbb{Z}+\frac{1}{2}\mathbb{Z}}\mathfrak{L}(k, t).$

Thanks to \eqref{eq:iit}, if $t\neq 0$, then $H=\mathfrak{L}(0, t)$, and if $t=0$, then $\mathfrak{L}(0, 0)=\mathfrak{L}.$
 The result 2) follows.
\end{proof}

\begin{defn}
Let $\delta$ be a linear mapping of $\mathfrak{L}$. If $\delta$ satisfies that $\delta\in Der(A,\cdot)\cap Der(A,[,,])$,
that is, for all $l_i,  m_i, r_i\in \mathbb{Z}+\frac{1}{2}\mathbb{Z},  1\leq i\leq 3,$
\begin{eqnarray}
\label{eq:multD}  &&\delta(L_{l_1,m_1}^{r_1}\cdot L_{l_2,m_2}^{r_2})=\delta(L_{l_1,m_1}^{r_1})\cdot L_{l_2,m_2}^{r_2}+L_{l_1,m_1}^{r_1}\cdot \delta(L_{l_2,m_2}^{r_2}),\\
&&\delta([L_{l_1,m_1}^{r_1},L_{l_2,m_2}^{r_2},L_{l_3,m_3}^{r_3}])=[\delta(L_{l_1,m_1}^{r_1}),L_{l_2,m_2}^{r_2}, L_{l_3,m_3}^{r_3}]\nonumber\\
&&+[L_{l_1,m_1}^{r_1},\delta(L_{l_2,m_2}^{r_2}), L_{l_3,m_3}^{r_3}]+[L_{l_1,m_1}^{r_1},L_{l_2,m_2}^{r_2}, \delta(L_{l_3,m_3}^{r_3})],\label{eq:3_LieD}
\end{eqnarray}
  then $\delta$ is called {\bf a derivation of $\mathfrak{L}$}.
\end{defn}

\begin{theorem} The left multiplications $ad(L_{l_1,m_1}^{r_1},L_{l_2,m_2}^{r_2}):\mathfrak{L}\rightarrow \mathfrak{L}$, for all~$\alpha_i=(l_i, m_i, r_i)\in (\mathbb{Z}+\frac{1}{2}\mathbb{Z})^3, 1\leq i\leq 3$ are  derivations of $\mathfrak{L}$.
\end{theorem}
\begin{proof}  Apply Eqs \eqref{eq:Jaco}, \eqref{eq:left} and \eqref{eq:muit}.
\end{proof}

\begin{lemma}\label{lem:basis}  $\mathfrak{L}$ has the minimal set of generators
$$S=\{L_{0,0}^{-1},L_{-1,0}^{0},L_{0,-1}^{0},L_{0,0}^{\frac{1}{2}},L_{\frac{1}{2},0}^{0},L_{0,\frac{1}{2}}^{0}\}.$$
\end{lemma}

\begin{proof}  By \eqref{eq:muit},  $(L_{l,m}^{r})^j=\underbrace {L_{l,m}^{r}\cdot L_{l,m}^{r} \cdot...\cdot L_{l,m}^{r}}_j =L_{jl,jm}^{jr}$, $\forall r,l,m\in \mathbb{Z}+\frac{1}{2}\mathbb{Z},~j\in \mathbb{Z}_{\geq0}$,~then
$$~~L_{l,0}^{0}=\left\{
     \begin{array}{ll}
       (L_{-1,0}^{0})^{-l}, & ~~~\hbox{$~l\in -\mathbb{Z}^+;$} \\
       (L_{-1,0}^{0}\cdot L_{\frac{1}{2},0}^{0})^{-2l}, & ~~\hbox{$~l\in -\mathbb{Z}^++\frac{1}{2};$} \\
       (L_{-1,0}^{0})\cdot (L_{\frac{1}{2},0}^{0})^{2}, & ~~\hbox{$~l=0;$} \\
       (L_{\frac{1}{2},0}^{0})^{2l}, & ~~\hbox{$~l\in \mathbb{Z}^+;$} \\
       (L_{\frac{1}{2},0}^{0})^{2l}, & ~~~\hbox{$~l\in \mathbb{Z}^+-\frac{1}{2}.$}
     \end{array}
   \right.
$$
$$~~~~~~~L_{0,m}^{0}=\left\{
     \begin{array}{ll}
       (L_{0,-1}^{0})^{-m}, & \hbox{$m\in -\mathbb{Z}^+;$} \\
       (L_{0,-1}^{0}\cdot L_{0,\frac{1}{2}}^{0})^{-2m}, & \hbox{$m\in -\mathbb{Z}^++\frac{1}{2};$} \\
       (L_{0,-1}^{0})\cdot (L_{0,\frac{1}{2}}^{0})^{2}, & \hbox{$m=0;$} \\
       (L_{0,\frac{1}{2}}^{0})^{2m}, & \hbox{$m\in \mathbb{Z}^+;$} \\
       (L_{0,\frac{1}{2}}^{0})^{2m}, & \hbox{$m\in \mathbb{Z}^+-\frac{1}{2}$.}
     \end{array}
   \right.
$$
$$~~~L_{0,0}^{r}=\left\{
     \begin{array}{ll}
       (L_{0,0}^{-1})^{-r}, & \hspace{2mm}~~~~\hbox{$r\in -\mathbb{Z}^+;$} \\
       (L_{0,0}^{-1}\cdot L_{0,0}^{\frac{1}{2}})^{-2r}, & \hbox{$~~~~r\in -\mathbb{Z}^++\frac{1}{2};$}~~~\\
       (L_{0,0}^{-1})\cdot (L_{0,0}^{\frac{1}{2}})^{2}, & ~~~~\hbox{$r=0;$} \\
       (L_{0,0}^{\frac{1}{2}})^{2r}, & ~~~~\hbox{$r\in \mathbb{Z}^+;$} \\
       (L_{0,0}^{\frac{1}{2}})^{2r}, & ~~~~\hbox{$r\in \mathbb{Z}^+-\frac{1}{2}$.}
     \end{array}
   \right.~~
$$
Thanks to  $L_{l,m}^{r}=L_{l,0}^{0}\cdot L_{0,m}^{0}\cdot L_{0,0}^{r}$, we get  $S=\{L_{0,0}^{-1},L_{-1,0}^{0},L_{0,-1}^{0},L_{0,0}^{\frac{1}{2}},L_{\frac{1}{2},0}^{0},L_{0,\frac{1}{2}}^{0}\}$ is a minimal set of generators of $\mathfrak{L}$.
\end{proof}

\begin{theorem}\label{thm:Dermulti}
Derivation algebra~$Der(A,\cdot)$~of commutative associative algebra$(A,\cdot)$ is spanned by
 \begin{equation}\label{eq:DerA_cdot}
   \{\delta\mid\delta\in End(A),\delta(L_{l,m}^{r})=2rL_{l,m}^{r-\frac{1}{2}}\delta (L_{0,0}^{\frac{1}{2}})+2lL_{l-\frac{1}{2},m}^{r}\delta (L_{\frac{1}{2},0}^{0})+2mL_{l,m-\frac{1}{2}}^{r}\delta (L_{0,\frac{1}{2}}^{0})\}.
 \end{equation}
\end{theorem}

\begin{proof}
By \eqref{eq:multD},
\begin{equation*}
  \delta (L_{0,0}^{a})=\left\{
  \begin{array}{ll}
    -aL_{0,0}^{a+1}\delta (L_{0,0}^{-1}), & \hbox{$a\in-\mathbb{Z}^+$;} \\\\
    (\frac{1}{2}-a)L_{0,0}^{a+1}\delta (L_{0,0}^{-1})+L_{0,0}^{a-\frac{1}{2}}\delta (L_{0,0}^{\frac{1}{2}}), & \hbox{$a\in-\mathbb{Z}^++\frac{1}{2}$;} \\\\
    2aL_{0,0}^{a-\frac{1}{2}}\delta (L_{0,0}^{\frac{1}{2}}), & \hbox{$a\in \mathbb{Z}^+$;} \\\\
    (2a-1)L_{0,0}^{a-\frac{1}{2}}\delta
    (L _{0,0}^{\frac{1}{2}})+L_{0,0}^{a-\frac{1}{2}}\delta (L_{0,0}^{\frac{1}{2}}), & \hbox{$a\in \mathbb{Z}^++\frac{1}{2}$;} \\\\
    0, & \hbox{$a=0$,}
  \end{array}
\right.
\end{equation*}
and
\begin{equation*}
   \begin{split}
    0=\delta(L_{0,0}^{-a}\cdot L_{0,0}^{a})&=L_{0,0}^{a}\delta (L_{0,0}^{-a})+L_{0,0}^{-a}\delta (L_{0,0}^{a}) \\
       &=L_{0,0}^{a}aL_{0,0}^{-a+1}\delta (L_{0,0}^{-1})+2aL_{0,0}^{-a}L_{0,0}^{a-\frac{1}{2}}\delta (L_{0,0}^{\frac{1}{2}})\\
       &=aL_{0,0}^{1}\delta (L_{0,0}^{-1})+2aL_{0,0}^{-\frac{1}{2}}\delta (L_{0,0}^{\frac{1}{2}}).
   \end{split}
\end{equation*}
Then we have
\begin{equation*}
  \delta (L_{0,0}^{-1})=-2L_{0,0}^{-\frac{3}{2}}\delta (L_{0,0}^{\frac{1}{2}}), ~~ \mbox{ and} ~~
   \delta (L_{0,0}^{a})=2aL_{0,0}^{a-\frac{1}{2}}\delta (L_{0,0}^{\frac{1}{2}}), \forall~ a\in\mathbb{Z}+\frac{1}{2}\mathbb{Z}.
\end{equation*}
Similarly, we have
\begin{equation*}
  \delta (L_{a,0}^{0})=2aL_{a-\frac{1}{2},0}^{0}\delta (L_{\frac{1}{2},0}^{0}), \quad
  \delta (L_{0,a}^{0})=2aL_{0,a-\frac{1}{2}}^{0}\delta (L_{0,\frac{1}{2}}^{0}), ~~\forall~ a\in\mathbb{Z}+\frac{1}{2}\mathbb{Z}.
\end{equation*}
Thanks to ~$\delta\in Der(A,\cdot)$,  we have
\begin{equation*}
\begin{split}
  \delta(L_{l,m}^{r})&=\delta(L_{0,0}^{r}\cdot L_{l,0}^{0}\cdot L_{0,m}^{0})\\
&=\delta (L_{0,0}^{r})\cdot L_{l,0}^{0}\cdot L_{0,m}^{0}+L_{0,0}^{r}\cdot \delta (L_{l,0}^{0})\cdot L_{0,m}^{0}+L_{0,0}^{r}\cdot L_{l,0}^{0}\cdot \delta (L_{0,m}^{0})\\
&=2rL_{l,m}^{r-\frac{1}{2}}\delta (L_{0,0}^{\frac{1}{2}})+2lL_{l-\frac{1}{2},m}^{r}\delta (L_{\frac{1}{2},0}^{0})+2mL_{l,m-\frac{1}{2}}^{r}\delta (L_{0,\frac{1}{2}}^{0}).
\end{split}
\end{equation*}
It follows the result.\end{proof}

By Lemma \ref{lem:basis} and Theorem \ref{thm:Dermulti}, for discussing derivations of $\mathfrak{L}$, we need ~to study the properties of $\delta (L_{0,0}^{\frac{1}{2}})$~,~$\delta (L_{\frac{1}{2},0}^{0})$~and~$\delta (L_{0,\frac{1}{2}}^{0})$, where $\delta\in Der(\mathfrak{L})$. So, for any $\delta\in Der(\mathfrak{L})$,~suppose
\begin{equation}\label{eq:setD1}
  \delta (L_{0,0}^{\frac{1}{2}})=\sum_{r,l,m\in\mathbb{Z}+\frac{1}{2}\mathbb{Z}} \lambda(\frac{1}{2},0,0;r,l,m)L_{l,m}^{r},\quad \lambda(\frac{1}{2},0,0;r,l,m)\in \mathbb{F},
\end{equation}
\begin{equation}\label{eq:setD2}
  \delta (L_{\frac{1}{2},0}^{0})=\sum_{r,l,m\in\mathbb{Z}+\frac{1}{2}\mathbb{Z}} \lambda(0,\frac{1}{2},0;r,l,m)L_{l,m}^{r}, \quad \lambda(0,\frac{1}{2},0;r,l,m)\in \mathbb{F},
\end{equation}
\begin{equation}\label{eq:setD3}
  \delta (L_{0,\frac{1}{2}}^{0})=\sum_{r,l,m\in\mathbb{Z}+\frac{1}{2}\mathbb{Z}} \lambda(0,0,\frac{1}{2};r,l,m)L_{l,m}^{r}, \quad \lambda(0,0,\frac{1}{2};r,l,m)\in \mathbb{F}.
\end{equation}

\begin{theorem}\label{thm:delta}
If~$\delta\in Der(A,\cdot)$,~ then $\delta$ is a derivation of $\mathfrak{L}$ if and only if~

$$\lambda(\frac{1}{2},0,0;r,l,m), \quad\lambda(0,\frac{1}{2},0;r,l,m)~ ~~\mbox{and}~~~ \lambda(0,0,\frac{1}{2};r,l,m)$$
 satisfies the following identities
\end{theorem}
\begin{equation}\label{eq:GD}
   \left\{\begin{split}
&\sum_{l,m,r}\lambda(\frac{1}{2},0,0;r,l,m)(r-\frac{1}{2})+\lambda(0,\frac{1}{2},0;r-\frac{1}{2},l+\frac{1}{2},m)(l+\frac{1}{2})\\
+&\lambda(0,0,\frac{1}{2};r-\frac{1}{2},l,m+\frac{1}{2})(m+\frac{1}{2})=0,\\
&\sum_{r,l,m}\lambda(\frac{1}{2},0,0;r+\frac{1}{2},l-\frac{1}{2},m)(r-\frac{1}{2})+\lambda(0,\frac{1}{2},0;r,l,m)(l+\frac{1}{2})\\
+&\lambda(0,0,\frac{1}{2};r,l-\frac{1}{2},m+\frac{1}{2})(m+\frac{1}{2})=0,\\
&\sum_{r,l,m}\lambda(\frac{1}{2},0,0;r+\frac{1}{2},l,m-\frac{1}{2})(r-\frac{1}{2})
+\lambda(0,\frac{1}{2},0;r,l+\frac{1}{2},m-\frac{1}{2})(l+\frac{1}{2})\\
+&\lambda(0,0,\frac{1}{2};r,l,m)(m+\frac{1}{2})=0,~~\forall~r,l,m\in\mathbb{Z}+\frac{1}{2}\mathbb{Z}.\\
  \end{split}\right.
\end{equation}

\begin{proof}
For~any $\delta\in Der(A,\cdot)$,~thanks to  Theorem \ref{thm:Dermulti}, $\delta$ is a derivation of $\mathfrak{L}$ if and only if $\delta$  satisfies
\begin{eqnarray}
&&\delta([L_{0,0}^{\frac{1}{2}},L_{l_2,m_2}^{r_2},L_{l_3,m_3}^{r_3}])
            - [\delta (L_{0,0}^{\frac{1}{2}}),L_{l_2,m_2}^{r_2},L_{l_3,m_3}^{r_3}]
-[L_{0,0}^{\frac{1}{2}},\delta (L_{l_2,m_2}^{r_2}),L_{l_3,m_3}^{r_3}]\nonumber\\
&&-[L_{0,0}^{\frac{1}{2}},L_{l_2,m_2}^{r_2},\delta (L_{l_3,m_3}^{r_3})]=0,\label{eq:Drlm1}\\
&&\delta([L_{\frac{1}{2},0}^{0},L_{l_2,m_2}^{r_2},L_{l_3,m_3}^{r_3}])
            - [\delta (L_{\frac{1}{2},0}^{0}),L_{l_2,m_2}^{r_2},L_{l_3,m_3}^{r_3}]
-[L_{\frac{1}{2},0}^{0},\delta (L_{l_2,m_2}^{r_2}),L_{l_3,m_3}^{r_3}]\nonumber\\
&&-[L_{\frac{1}{2},0}^{0},L_{l_2,m_2}^{r_2},\delta (L_{l_3,m_3}^{r_3})]=0,\label{eq:Drlm2}\\
&&\delta([L_{0,\frac{1}{2}}^{0},L_{l_2,m_2}^{r_2},L_{l_3,m_3}^{r_3}])
- [\delta (L_{0,\frac{1}{2}}^{0}),L_{l_2,m_2}^{r_2},L_{l_3,m_3}^{r_3}]-[L_{0,\frac{1}{2}}^{0},\delta (L_{l_2,m_2}^{r_2}),L_{l_3,m_3}^{r_3}]\nonumber\\
&&-[L_{0,\frac{1}{2}}^{0},L_{l_2,m_2}^{r_2},\delta (L_{l_3,m_3}^{r_3})]=0,~~\forall~l_2,m_2,r_2,l_3,m_3,r_3,\in\mathbb{Z}+\frac{1}{2}\mathbb{Z}.\label{eq:Drlm3}
\end{eqnarray}

We get that  Eq \eqref{eq:Drlm1} holds if and only if
\begin{equation}\label{eq:001}
\begin{split}
\left|
    \begin{array}{cc}
    l_2 & l_3 \\
    m_2 & m_3 \\
    \end{array}
\right|&\big(\sum_{r,l,m}\lambda(\frac{1}{2},0,0;r,l,m)(r-\frac{1}{2})L_{l+l_2+l_3-1,m+m_2+m_3-1}^{r+r_2+r_3}\\
+&\sum_{r,l,m}\lambda(0,\frac{1}{2},0;r,l,m)(l+\frac{1}{2})L_{l+l_2+l_3-\frac{3}{2},m+m_2+m_3-1}^{r+r_2+r_3+\frac{1}{2}}\\
+&\sum_{r,l,m}\lambda(0,0,\frac{1}{2};r,l,m)(m+\frac{1}{2})L_{l+l_2+l_3-1,m+m_2+m_3-\frac{3}{2}}^{r+r_2+r_3+\frac{1}{2}}\big)=0.\\
\end{split}
\end{equation}
Therefore, if ~$\left|
    \begin{array}{cc}
    l_2 & l_3 \\
    m_2 & m_3 \\
    \end{array}
\right|= 0$,~ \eqref{eq:Drlm1} holds. If~$\left|
    \begin{array}{cc}
    l_2 & l_3 \\
    m_2 & m_3 \\
    \end{array}
\right|\neq 0$, we have
\begin{equation}\label{eq:0012}
  \begin{split}
&\sum_{l,m,r}\big(\lambda(\frac{1}{2},0,0;r,l,m)(r-\frac{1}{2})+\lambda(0,\frac{1}{2},0;r-\frac{1}{2},l+\frac{1}{2},m)(l+\frac{1}{2})\\
+&\lambda(0,0,\frac{1}{2};r-\frac{1}{2},l,m+\frac{1}{2})(m+\frac{1}{2})\big)L_{l+l_2+l_3-1,m+m_2+m_3-1}^{r+r_2+r_3}=0.
  \end{split}
\end{equation}
By the similar discussion to the above, Eqs \eqref{eq:Drlm2} and \eqref{eq:Drlm3} hold if and only if
\begin{equation}\label{eq:0022}
  \begin{split}
&\sum_{r,l,m}\big(\lambda(\frac{1}{2},0,0;r+\frac{1}{2},l-\frac{1}{2},m)(r-\frac{1}{2})+\lambda(0,\frac{1}{2},0;r,l,m)(l+\frac{1}{2})\\
+&\lambda(0,0,\frac{1}{2};r,l-\frac{1}{2},m+\frac{1}{2})(m+\frac{1}{2})\big)L_{l+l_2+l_3-1,m+m_2+m_3-1}^{r+r_2+r_3}=0,
  \end{split}
\end{equation}
\begin{equation}\label{eq:0032}
  \begin{split}
&\sum_{r,l,m}\big(\lambda(\frac{1}{2},0,0;r+\frac{1}{2},l,m-\frac{1}{2})(r-\frac{1}{2})
+\lambda(0,\frac{1}{2},0;r,l+\frac{1}{2},m-\frac{1}{2})(l+\frac{1}{2})\\
+&\lambda(0,0,\frac{1}{2};r,l,m)(m+\frac{1}{2})\big)L_{l+l_2+l_3-1,m+m_2+m_3-1}^{r+r_2+r_3}=0,\\
  \end{split}
\end{equation}
hold, respectively. Thanks to $L_{l+l_2+l_3-1,m+m_2+m_3-1}^{r+r_2+r_3}\neq 0$, and Eqs \eqref{eq:0012}, \eqref{eq:0022} and \eqref{eq:0032}, we get that $\delta$ is a derivation of $\mathfrak{L}$ if and only if~Eq \eqref{eq:GD} holds.
\end{proof}

\bibliography{}

\end{document}